\documentclass[draft, reqno]{amsart}
\usepackage[all]{xy}                        %

\CompileMatrices                            

\UseTips                                    

\input xypic
\usepackage[bookmarks=true]{hyperref}       

\usepackage{amssymb,latexsym,amsmath,amscd}
\usepackage{xspace}
\usepackage{enumerate}
\usepackage{graphicx}
\usepackage{vmargin}
\usepackage{todonotes}

\reversemarginpar

\theoremstyle{plain}
\newtheorem{theorem}{Theorem}[section]
\newtheorem*{theorem*}{Theorem}
\newtheorem{proposition}[theorem]{Proposition}
\newtheorem{corollary}[theorem]{Corollary}
\newtheorem{lemma}[theorem]{Lemma}

\theoremstyle{definition}
\newtheorem{definition}[theorem]{Definition}

\newtheorem{remark}[theorem]{Remark}
\newtheorem{example}[theorem]{Example}

\newcommand{\enm}[1]{\ensuremath{#1}}          %

\newcommand{\cal}[1]{\mathcal{#1}}

\newcommand{\NN}{\enm{\mathbb{N}}}

\newcommand{\PP}{\enm{\mathbb{P}}}

\newcommand{\Ii}{\enm{\cal{I}}}

\newcommand{\Oo}{\enm{\cal{O}}}

\newcommand{\Ss}{\enm{\cal{S}}}

\renewcommand{\phi}{\varphi}
\renewcommand{\theta}{\vartheta}
\renewcommand{\epsilon}{\varepsilon}

\begin{document}

\title[Segre varieties]
{Linearly dependent and concise subsets of a Segre variety depending on $k$ factors}
\author{E. Ballico}
\address{Dept. of Mathematics\\
 University of Trento\\
38123 Povo (TN), Italy}
\email{ballico@science.unitn.it}
\thanks{The author was partially supported by MIUR and GNSAGA of INdAM (Italy).}

\begin{abstract}
We study linearly dependent subsets with prescribed cardinality, $s$, of a multiprojective space. If the set $S$ is a circuit, we give
an upper bound on the number of factors of the minimal multiprojective space containing $S$, while if $S$ has higher dependency this may be not true without strong assumptions.
We describe the dependent subsets $S$ with $\#S=6$.
\end{abstract}

\maketitle

\section{Introduction}
Take $k$ non-zero finite dimensional vector spaces $V_1,\dots,V_k$ and consider $V_1\otimes \cdots \otimes V_k$. An element
$u\in V_1\otimes \cdots \otimes V_k$ is called a $k$-tensor with format $(\dim V_1,\dots ,\dim V_k)$ (\cite{l}).
Two non-zero proportional tensors share many properties. Thus often the right object to study is the
projectivization, $\PP^r$, of $V_1\otimes \cdots \otimes V_k$, where $r:= -1 +\dim V_1\times \cdots \times \dim V_k$. Set
$n_i:= \dim V_i -1$ and consider the multiprojective space $Y:= \PP^{n_1}\times \cdots \times \PP^{n_k}$. Let $\nu:
Y\hookrightarrow \PP^r$ denote the Segre embedding. Many properties of a non-zero tensor  $u$ (e.g., the tensor rank and the tensor
border rank) may be describe in how its equivalence class $[u]\in \PP^r$ sits with respect to the Segre variety $\nu (Y)$. For instance, the tensor rank $r_Y([u])$ of
$u$ is the minimal cardinality of a finite set $S\subset Y$ such that $\nu (S)$ spans $[u]$. We call $\Ss (Y,[u])$ the set of
all $S\subset Y$ with minimal cardinality such that $\nu (S)$ spans $[u]$. Using subsets of $Y$ instead of ordered sets of
points and $\PP^r$ instead of  $V_1\otimes \cdots \otimes V_k$ we take care of the obvious non-uniqueness in a finite finite
decomposition $u = \sum _i v_{i1}\otimes \cdots v_{ik}$, $v_{ij}\in V_j$, of a tensor.

Fix an equivalence class $q =[u]\in \PP^r$ of non-zero tensors. The \emph{width} $w(q)$ of $q$ is the minimal number of
non-trivial factors of the minimal multiprojective subspace $Y'\subseteq Y$ such that $q\in \langle \nu (Y')\rangle$, where $\langle \ \ \rangle$ denote the linear span. For any
finite set
$A\subset Y$ the
\emph{width} $w(A)$ of $A$ is the number of integers $i\in
\{1,\dots ,k\}$ such that $\sharp (\pi _i(A)) >1$. By concision we have $w(q) = w(A)$ if $A\in \Ss (Y,q)$ (\cite[Proposition 3.2.2.2]{l}).

The non-uniqueness of tensor decompositions, i.e. the fact that $\Ss (Y,[u])$ may have more than one element, may be rephrased
as the linear dependency of certain subsets of $Y$ 8\cite{bbs}). For any finite set
$S\subset Y$ set
$e(S):= h^1(\Ii _S(1,\dots ,1))$. By the definition of Segre embedding and the Grassmann's formula we have
$e(S) =
\#S -1 -\dim \langle \nu (S)\rangle$. We say that a non-empty finite set
$S\subset Y$ (or that
the finite set $\nu (S)\subset
\PP^r$) is
\emph{equally dependent} if
$\dim
\langle
\nu (S)\rangle \le \sharp (S) -2$ and $\langle \nu (S')\rangle = \langle \nu (S)\rangle$ for all $S'\subset S$ such that
$\sharp(S') =\sharp (S)-1$. Note that $S$ is equally dependent if and only if $e(S)>0$ and $e(S')<e(S)$ for all
$S'\subsetneq S$, i.e. if and only if $S\ne \emptyset$ and $e(S') <e(S)$ for all $S'\subsetneq S$. We say that $S$ is \emph{uniformly dependent}  if $e(S') =\max \{0,e(S)-\#S+\#S'\}$ for all $S'\subset S$.
A uniformly dependent subset is equally dependent, but when $e(S)\ge 2$ the two notions have quite different (the key Examples
\ref{k2} and \ref{k3} are equally dependent, but not uniformly dependent). When $e(S)=1$ equal and uniform dependence
coincide. An equally dependent subset with $e(S)=1$ is often called a \emph{circuit}.

\begin{theorem}\label{x1}
Let $S\subset Y$ be a circuit. Set $s:= \# (S)$. Then $w(S)\le \binom{s}{2} +s$.
\end{theorem}

We give examples for any $\#S \ge 6$ of an equally dependent set $S$ with $e(S)>1$ and $w(S)$ arbitrarily large (Example
\ref{k4}). Thus Theorem
\ref{x1} is not true for equally dependent subsets with higher dependency $e(S)$.

To see that Theorem \ref{x1} may be used in cases with $e(S)>1$ we consider the following extremal case.
Fix an integer $e>0$. Let $S$ be a finite subset of a multiprojective space. We say that $S$ is an \emph{$e$-circuit} if $e(S)=e$
and there is a subset $S'\subseteq S$ such that $S'$ is a circuit and $\#S-\#S' =e-1$. A uniformly dependent set $S$ is an
$e(S)$-circuit, but the converse does not hold (Example \ref{z1}).

\begin{corollary}\label{x1.1}
Let $S\subset Y$ be an $e$-circuit. Set $z:= \# (S)-e+1$. Then $w(S)\le \binom{z}{2}+z$.
\end{corollary}

In the second part of this paper we study equally dependent subsets $S$ of a Segre variety with $\#S=6$. We prove the
following result.

\begin{theorem}\label{is1}
Let $Y = \PP^{n_1}\times \cdots \times \PP^{n_k}$, $n_1\ge \cdots \ge n_k>0$ be a multiprojective space and $S\subset Y$
a concise and equally dependent set with $\#S=6$. Then either $e(S) \ge 2$ and $(Y,S)$ is in one of the Examples \ref{k2} and \ref{k3} or $w(Y) \le 4$ and $Y =(\PP^1)^4$ if $w(Y)=4$.\end{theorem}
The families in Examples \ref{k2}, \ref{k3} have arbitrarily large width. The case $Y =(\PP^1)^k$ and $e(S) =1$ occurs (\cite[case 3 of Theorem 7.1]{bbs}).
In several cases we could give a more precise description of the pairs $(Y,S)$, but using too much ink (unless some description is needed for the proof of Theorem \ref{is1}).

For any $q\in \PP^r$ and any finite set $S\subset Y$ we say that \emph{$S$ irredundantly spans $q$} if $q\in \langle
\nu(S)\rangle$ and $q\notin \langle \nu (S')\rangle$ for any $S'\subsetneq S$. To prove Theorem \ref{is1} we also
classify the set of all rank
$2$ tensors which may be irredundantly spanned by a set of
$3$ points (Proposition \ref{cp1}).

We work over a field $K$, since for the examples we only use that $\PP^1(K)$ has at least $3$ points. For the proofs which require cohomology proofs (like in the quotations
of \cite[5.1]{bb2} or \cite[2.4, 2.5]{bbcg1} it is sufficient to work over the algebraic closure $\overline{K}$ of $K$, because dimensions of cohomology of algebraic sheaves  on projective varieties (and in particular the definition of $e(S)$) are invariants under the extension $K\hookrightarrow \overline{K}$ (\cite[Proposition III.9.3]{h}). We use Landsberg's book \cite{l} for essential properties on Segre varieties related to tensors (e.g., the notion of concision). This book contains many applications of tensors and additive tensor decompositions are just a way to state linear combinations of elements of the Segre variety $\nu (Y)$. The elementary properties of the Segre varieties that we use do not depends on the base field. For an detailed study of them over a finite field, see \cite{ht}.

\subsection{Outline of the proof of Theorem \ref{is1}}

In section \ref{Sk} we describe the examples mentioned in the statement of Theorem \ref{is1}.
Take $S\subset Y$ such that $\#S=6$ and $S$ is equally dependent. We fix a partition $S =A\cup B$ with $\#A =\#B =3$ and hence $A\cap B =\emptyset$.
In section \ref{So} we assume that at least one among $\nu (A)$ and $\nu(B)$ is linearly dependent. In that section 
we get Examples \ref{k2} and \ref{k3}. Then we assume $\nu (A)$ and $\nu (B)$ linearly independent. Since $A\cap B =\emptyset$, the Grassmann's
formula gives $\dim \langle \nu (A)\rangle \cap \langle \nu (B)\rangle =e(S)-1$. Thus $\langle \nu (A)\rangle \cap \langle \nu
(B)\rangle \ne \emptyset$. We fix a general $q\in \langle \nu (A)\rangle \cap \langle \nu (B)\rangle$. Since $q\in \langle
\nu(A)\rangle$, we have $r_Y(q)\le 3$. We discuss the cases $r_Y(q)=1$, $r_Y(q)=2$,
$r_Y(q)=3$ in sections \ref{Su}, \ref{Sd} and \ref{St}, respectively. For the case $r_Y(q)=3$ we use \cite[Theorem 7.1]{bbs}.
\begin{remark}\label{o8}
The case $k=1$ is  possible with  $Y=\PP^n$ for any $2\le n \le 4$ (any $6$ points spanning  $\PP^n$ partitioned
in two sets of $3$ element no $3$ of them collinear). The case $Y =\PP^1$ was obtained when $e(A)>0$ and $e(B)>0$. When
$Y=\PP^n$ we have $e(S) =6-n-1$.
\end{remark} Thus in sections \ref{So}, \ref{Su}, \ref{Sd} and \ref{St} we silently assume
$k>1$.

\section{Preliminaries and notation}
For any subset $E$ of any projective space let $\langle E\rangle$ denote the linear span of $E$.

For any multiprojective space let
$\nu$ denote its Segre embedding. Let
$Y =\PP^{n_1}\times
\cdots
\times
\PP^{n_k}$ be a multiprojective space. Let $\pi _i: Y\to \PP^{n_i}$ be the projection of $Y$ onto its $i$-th factor. Set $Y_i
:=\prod _{j\ne i} \PP^{n_j}$ and let $\eta _i: Y\to Y_i$ the projection. Thus for any $p= (p_1,\dots ,p_k)\in Y$, $\pi _i({p})
=p_i$ is the
$i$-th component of $p$, while $\eta _i({p}) =(p_1,\dots ,p_{i-1},p_{i+1},\dots ,p_k)$ deletes the $i$-th component of $p$. For
any set $E\subsetneq \{1,\dots ,k\}$ set $Y_E:= \prod _{i\in \{1,\dots ,k\}\setminus E} \PP^{n_i}$ and let $\pi _E: Y\to Y_E$ the projection which forgets all coordinates $i\in E$.

For any $i\in \{1,\dots ,k\}$ let $\epsilon _i\in \NN^k$ (resp. $\hat{\epsilon}_i)$ be the multiindex $(a_1,\dots ,a_k)\in \NN^k$ with $a_i=1$ and $a_h=0$ for all $h\ne i$ (resp. $a_i=0$ and $a_h=1$ for all $h\ne i$). Thus $\Oo _Y(\epsilon _i)$ and $\mathcal {O}_Y(\hat{\epsilon}_i)$ are line bundles on $Y$ and $\Oo _Y(\epsilon _i)\otimes \mathcal {O}_Y(\hat{\epsilon}_i)\cong \Oo _Y(1,\dots ,1)$.

If needed we usually call $\PP^r$ the projectivization of the space of tensors with prescribed format we are working, i.e. the projective space in which the given Segre sits.
For instance, if the given Segre is $\nu (Y)$ we take $r =-1 +\prod _{i=1}^{k}(n_i+1)$.  For any $q\in \PP^r$ e denote with $r_Y(q)$ or with
$r_{\nu (Y)}(q)$ the tensor rank of $q$. For any finite set
$A\subset Y$ the minimal multiprojective subspace of
$Y$ containing
$\prod _{i=1}^{k}
\langle
\pi _i(A)\rangle$. For any $q\in \PP^r:= \langle \nu (Y)\rangle$  let
$r_{\nu (Y)}(q)$ denotes the tensor rank of
$q$. For any positive integer $t$ let $\Ss
(Y,q,t)$ denote the set of all
$S\subset Y$ such that
$q\in
\langle
\nu (S)\rangle$,
$\#S = t$ and $S$ irredundantly spans $q$. The set $\Ss (Y,q):= \Ss (Y,q,r_{\nu (Y)}(q))$ is the set of all tensor decompositions
of $q$ with minimal length. By concision given any
$A\in \Ss (Y,q)$ the minimal multiprojective subspace of $Y$ containing $A$ is the minimal multiprojective subspace $Y'\subseteq Y$ such that $q\in \langle \nu (Y')\rangle $ (\cite[Proposition 3.2.2.2]{l}).

\begin{remark}\label{z3}
Take $S\subset Y$ such that $e(S)>0$ and $\#S\le 3$. Since $\nu$ is an embedding, we have $\#S=3$, $e(S)=1$ and (by the structure of linear subspaces contained in a Segre variety) there is $i$ such that $\pi _h(S)=1$ for all $h \ne 1$, $\pi _{i|S}$ is injective and $\pi _i(S)$ is contained in a line.
\end{remark}

\begin{lemma}\label{o4.1}
Fix a multiprojective space $Y$ and any finite set $Z\subset Y$ with $z:= \#Z \ge 3$ and concise for $Y$. Set $e(Z):= z-1 -\dim \langle \nu(Z)\rangle$. We have $e(Z)\le z-2$ and equality holds if and only if $Y=\PP^1$.
\end{lemma}

\begin{proof}
Since $\nu$ is an embedding, $\nu (Z)$ is a set of $z\ge 2$ points of $\PP^N$ and hence $\dim \langle \nu(Z)\rangle \ge 1$. The Grassmann's formula gives $e(Z)\le z-2$
and that equality holds if and only if $\nu (Z)$ is formed by collinear points. Since $z\ge 3$, the Segre $\nu(Y)$ is cut out by quadrics and $z\le 3$, we get $\langle \nu (Z)\rangle \subseteq \nu(Y)$. Since the lines of a Segre variety are Segre varieties, the concision assumption gives $Y=\PP^1$.

The converse is trivial, because $h^0(\Oo _{\PP^1}(1))=2$.\end{proof}

The following construction was implicitely used in \cite{bbcg}.

\begin{definition}Fix a multiprojective space $Y =\PP^{n_1}\times \cdots \times \PP^{n_k}$, $n_h>0$ for all $h\ne i$, and $i\in \{1,\dots ,k\}$ (we allow the case $n_i=0$ so that $\PP^{n_i}$ may be a single point). Fix an integer $m_i$ such that $n_i\le m_i\le n_i+1$; if $n_i=0$ assume $m_i=1$. Let $W\supseteq Y$ be a multiprojective space with $\PP^{n_j}$ as its $j$-th factor for all $j\ne i$ and with $\PP^{m_i}$ as its $i$-th factor. Thus $W=Y$ if $m_i=n_i$ and $\dim W =\dim Y+1$ if $m_i=n_i+1$.
If $W\ne Y$ we identify $Y$ with a multiprojective subspace of $W$ identifying its factor $\PP^{n_i}$ with a hyperplane
$M_i\subset \PP^{m_i}$. Fix a finite set $E\subset Y$ (we allow the case $E =\emptyset$) and $o =(o_1,\dots ,o_k)\in
Y\setminus \{o\}$. Set $E_i:= \pi _i(E)\subset \PP^{n_i}$. Fix any $u_i\in \PP^{m_i}\setminus (E_i\cup \{o_i\})  $ and any
$v_i\in \langle \{o_i,v_i\}\rangle$ with $v_i\notin E_i$. Set $u =(u_1,\dots ,u_k)$ and $v:= (v_1,\dots ,v_k)$ with
$u_h=v_h=o_h$ for all $h\ne i$. Set $F:= E\cup \{o\}$ and $G:= E\cup \{u,v\}$. We say that $G$ is an \emph{elementary
increasing} of $F$ with respect to $o$ and the $i$-th factor. Note that $\#G=\#E+2$, $\# F =\#E+1$ and $\langle \nu (F)\rangle
\subseteq \langle \nu (G)\rangle$. If $n_i>0$ we have $w(Y) =w(W)$, while if $n_i=0$ we have $w(W)=w(Y)+1$. Thus an elementary
operation may increase the width, but only by $1$ and only if $n_i=0$.
\end{definition}

\begin{remark}\label{z2}
Let $U\subset Y$ be a finite set, $W\supseteq Y$ any multiprojective space and $V\subset W$ any set obtained from $U$ making an
elementary increasing. For any finite set $G\subset W$ we have either $w(V\cup G) =w(U\cup G)$ or $w(V\cup G) =w(U\cup G)+1$, but the latter may occur only if
$w(V)=w(U)+1$. Even when $w(V) =w(U)+1$ it is quite easy to see for which $G$ we have $w(V\cup G)=w(U\cup G)+1$.
\end{remark}

\section{Proofs of Theorem \ref{x1} and Corollary \ref{x1.1}}

\begin{lemma}\label{a1}Fix a finite set $S\subset Y$, $S\ne \emptyset$. Set $s:= \# (S)$. Assume $k > \binom{s}{2}$. Then
there is
$i\in
\{1,\dots ,k\}$ such that $\eta _{i|S}$ is injective.
\end{lemma}

\begin{proof}
Since the lemma is trivial if $s=1$, we may assume $s\ge 2$. Assume that the lemma is false. Thus for every $i\in \{1,\dots
,k\}$
there is $\{a_i,b_i\}\subseteq S$ such that $a_i\ne b_i$ and $\eta _i(a_i) =\eta _i(b_i)$. Since $k> \binom{s}{2}$, by the
pigeonhole's principle there are $i, j\in \{1,\dots ,k\}$ and $a, b\in S$ such that $i\ne j$, $a\ne b$ and $\{a,b\}
=\{a_i,b_i\} = \{a_j,b_j\}$. Since $(\{1,\dots ,k\}\setminus \{i\})\cup (\{1,\dots ,k\}\setminus \{j\}) = \{1,\dots ,k\}$,
we get $\pi _h(a)=\pi _h(b)$ for all $h\in \{1,\dots ,k\}$. Thus $a=b$, a contradiction.
\end{proof}

\begin{lemma}\label{a2}
Fix a finite set $S\subset Y$, $S\ne \emptyset$. Set $s:= \# (S)$. Assume $k > \binom{s}{2}+s$. Then
there is
$E\subset
\{1,\dots ,k\}$ such that $\# (E) =s$ and $\eta _{E|S}$ is injective.
\end{lemma}

\begin{proof}
Since the lemma is trivial if $s=1$, may assume $s\ge 2$. By Lemma \ref{a1} there is $i\in \{1,\dots ,k\}$
such that $\eta _{i|S}$ is injective. Then we continue applying $s-1$ times Lemma \ref{a1} first to the set $\eta _i(S)$ and
the multiprojective space $Y_i$ and then each time to a multiprojective space with a smaller number of non-trivial factors.
\end{proof}

\begin{proof}[Proof of Theorem \ref{x1}:]
Assume $k:= w(S)> \binom{s}{2} +s$. Call $Y = \PP^{n_1}\times \cdots \times \PP^{n_k}$ the minimal multiprojective space containing $S$. We
use induction on the integer $s$, the cases with $s\le 3$ being true, because no such $S$ exists when $s \le 2$, while for $s=3$
we have $w(S)=1$ by the structure of linear subspaces contained in a Segre variety.  By Lemma \ref{a2} there is a set
$E\subset \{1,\dots ,k\}$ such that $\# (E)=s$ and $\eta_{E|S}$ is injective. With no loss of generality we may assume $E
=\{1,\dots ,s\}$.

\quad {\bf Observation 1:} Fix $i\in \{1,\dots ,k\}$ and any hyperplane $M\subset \PP^{n_i}$. Since $S$ is a circuit and $Y$
is the minimal multiprojective space containing $Y$, we have  $h^1(\Ii _{\pi _i^{-1}(M)\cap S}(1,\dots ,1)) =0$.

\quad (a) Fix $i\in
\{1,\dots ,s\}$ and $o\in \pi _i(S)\subset \PP^{n_i}$. Take a general hyperplane
$M\subset \PP^{n_i}$ containing $o$. Set $H:= \pi _i^{-1}(M)$. Since $\pi _i(S)$ is a finite set and $H$ is general, we have
$H\cap\pi _i(S) = \{o\}$. Set $S':=S\setminus S\cap H$. We have the residual exact sequence
\begin{equation}\label{eqb1=}
0 \to \Ii_{S'}(\hat{\epsilon}_i) \to \Ii_S(1,\dots ,1) \stackrel{v}{\to} \Ii_{S\cap H,H}(1,\dots ,1) \to 0
\end{equation} 
Since $S'$ is a finite set, the exact sequence
$$0 \to \Ii_{S'}(\hat{\epsilon}_i) \to \Oo_Y(\hat{\epsilon}_i) \to \Oo_{S'}(\hat{\epsilon}_i)\to 0$$ gives
$h^2(\Ii_{S'}(\hat{\epsilon}_i)) = h^2(\Oo_Y(\hat{\epsilon}_i))=0$. Since $h^1(H,\Ii_{S\cap H,H}(1,\dots ,1))=0$ by Observation 1,
\eqref{eqb1=} gives $h^1(\Ii_{S'}(\hat{\epsilon}_i))=1$. Since $\eta _{i|S}$ is injective and $S'\subset S$, $\eta _{i|S'}$ is
injective. Thus we have
$h^1(\Ii_{S'}(\hat{\epsilon}_i))=h^1(Y_i,\Ii_{\eta _i(S')}(1,\dots ,1))$.

\quad (a1) Assume that $\eta _i(S')$ is not equally dependent. Thus there is $U\subsetneq S'$ such that $c:= h^1(Y_i,\Ii_{\eta
_i(U)}(1,\dots ,1))>0$ and $\dim \mathrm{coker}(v) = c-1$. Let $v'$ the surjection in an exact sequence like (\ref{eqb1=}) with
$(S\cap H)\cup U$ instead $S$. Since $\mathrm{ker}(v')\subseteq \mathrm{ker}(v)$, we have $\dim \mathrm{coker}(v') \le 
\mathrm{coker}(v) = c-1$. Thus
$h^1(\Ii_{(S\cap H)\cup U}(1,\dots ,1))\ge 1$, contradicting the assumption that $S$ is a circuit.

\quad (a2) Assume that $\eta _i(S')$ is a circuit. Since $\# (\eta _i(S'))<\# (S)$, the inductive assumption
gives $w(\eta _i(S'))\le \binom{z}{2}+z$, where $z:= \# (S')$. Hence $w(S') \le 1+\binom{z}{2}+z$.

\quad (b) Let $\alpha$ be the minimal positive integer appearing among the integers $\# (\pi _i^{-1}(o)\cap S)$, $i\in \{1,\dots ,s\}$, $o\in \pi _i(S)$.

\quad (b1) Assume $\alpha =1$ and take $i\in \{1,\dots ,s\}$ and $o\in \pi _i(S)$ such that $\pi _i^{-1}(o)\cap S$ is a single point, $p$. Set $S':= S\setminus \{p\}$. By step
(a) with $z=s-1$ we have $w(S')\le s+\binom{s-1}{2}$. Since $k =w(S) > \binom{s}{2}+s$ by assumption, there is a subset $E\subset \{1,\dots ,k\}$ such that $\# (E) =s$, $\# (\pi _h(S')) =1$ and $\# (\pi _h(S)) =2$ for all $h\in E$. The minimality of $Y$ implies $n_h=1$ for all $h\in E$. Fix $h\in E$. Set $H':=  \pi _h^{-1}(\pi _h(S'))$. We have $S' =H'\cap S$ and hence we have a residual exact sequence
\begin{equation}\label{eqb2}
0 \to \Ii_p(\hat{\epsilon}_h) \to \Ii_S(1,\dots ,1) \to \Ii_{S',H}(1,\dots ,1) \to 0
\end{equation} 
Since $S$ is a circuit and $S'\subsetneq S$, we have $h^1(H,\Ii_{S',H}(1,\dots ,1))=0$. Since $\Oo _Y(\hat{\epsilon}_h)$ is globally generated, we have $h^1(\Ii_p(\hat{\epsilon}_h))=0$. Thus (\ref{eqb2}) gives $h^1(\Ii _S(1,\dots ,1))=0$, a contradiction.

\quad (b2) Assume $\alpha \ge 2$. 
Since $Y$ is minimal among the multiprojective spaces containing $S$, we have $\# (\pi _i(S))\ge n_i+1$ for all $i$ and
in particular $\alpha \le s/2$. For all $i\in \{1,\dots ,s\}$ and $o\in \pi _i(S)$ set $S_{i,o}:= S\setminus S\cap \pi
_i^{-1}(o)$.
Let $Y_{i,o}$ be the minimal multiprojective subspace of $Y$ containing $S_{i,o}$. Since $w(Y_{i,o}) \le
\binom{s-1}{2}-s+1\le k-s$ (Step (a)) there is $E\subset \{1,\dots ,k\}$ such that $\# (E) \ge s$ and for each $h\in
E$ we have
$\# (\pi _h(S_{i,o}))=1$. Fix $h\in E$ and set $H_h:=\pi _h^{-1}(\pi _h(S_{i,o}))$ and $S_h:= S\setminus S\cap
H_h$. We have
$S_h\subseteq S\setminus S_{i,o}$ and hence $\# (S_h)\le \alpha \le s/2$. We have $S_h\ne \emptyset$, because $Y$ is the
minimal multiprojective space containing $S$. The residual exact sequence of $H_h$ gives $h^1(\Ii _{S_h}(\hat{\epsilon}_h))
>0$. Thus either $\eta _{h|S_h}$ is not injective or $h^1(Y_h,\Ii _{\eta _h(S_h)}(1,\dots ,1)) >0$.

\quad (b2.1) Assume $h^1(Y_h,\Ii _{\eta _h(S_h)}(1,\dots ,1)) >0$. As in steps (a1) and (a2) we get that $\eta _h(S_h)$ is a circuit and that
$w(S_h) \le 1+\binom{x}{2}+x$, where $x =\# (S_h)\le s/2$. As in step (b1) we see that $x\ge 2$. We saw that $S_h$ (resp.
$S\setminus S_h$) depends on at most $1+x+\binom{x}{2}$ (resp. $1+ s-x + \binom{s-x}{2}$) factors of $Y$. Since $2\le x\le s/2$ and $k>
2+s+\binom{x}{2} +\binom{s-x}{2}$, there is
$j\in
\{1,\dots ,k\}$ such that
$\# (\pi _j(S_h)) =\# (\pi _j(S\setminus S_h))=1$. Since $Y$ is the minimal multiprojective space containing $S$,
the assumption implies $n_j=1$ and $\pi _j(S_h) \ne \pi _j(S\setminus S_h)$. The residual exact sequence of the divisor $\pi
_j^{-1}(\pi _j(S_h))$ gives $h^1(\Ii _{S\setminus S_h}(\hat{\epsilon}_j)) >0$. Since $\# (\pi _j(S\setminus S_h))=1$,
we have $h^1(\Ii _{S\setminus S_h}(\hat{\epsilon}_j)) =h^1(\Ii _{S\setminus S_h}(1,\dots ,1)) >0$, contradicting the
assumption that $S$ is a circuit.

\quad (b2.2) Assume that $\eta _{h|S_h}$ is not injective. Hence $\eta _{h|S_{i,o}}$ is not injective. By step (b2.1) we may
assume that this is true for all
$h\in E$. Thus for each $h\in E$ there are $u_h,v_h \in S_{i,o}$ such that $u_h\ne v_h$ and $\pi _j(u_h) = \pi _j(v_h)$
for all $j\in \{1,\dots ,k\}$. Since $\# (S_{i,o})\le \# (E)/2$, there are $j, h\in E$ such that $\{u_h,v_h\}\cap
\{u_j,v_j\} \ne \emptyset$. First assume $\# (\{u_h,v_h\}\cap
\{u_j,v_j\})=1$. Permuting the names of the points $u_h,v_h,u_i,v_j$ we may assume $u_h=u_j$ and $v_h\ne v_j$. Since
 $(\{1,\dots ,k\}\setminus \{h\})\cup (\{1,\dots ,k\}\setminus
\{j\}) =\{1,\dots ,k\}$ we get $\pi _t(v_h)=\pi _t(v_j)$ for all $t\in \{1,\dots ,k\}$. Thus $v_h = v_j$, a contradiction.
Now assume $\{u_h,v_h\}\cap
\{u_j,v_j\} =\{u_h,v_h\}$, say $u_h=u_j$ and $v_h=v_j$. Since $(\{1,\dots ,k\}\setminus \{h\})\cup (\{1,\dots ,k\}\setminus
\{j\}) =\{1,\dots ,k\}$ we get $\pi _t(v_h)=\pi _t(u_h)$ for all $t\in \{1,\dots ,k\}$. Thus $v_h = u_h$, a contradiction.
\end{proof}

\begin{proof}[Proof of Corollary \ref{x1.1}]
If $e=1$, then $z=s$ and we apply Theorem \ref{x1}. Assume $e>1$ and take $U\subset S$ such that $\# (U) =e-1$
and $S\setminus U$ is a circuit. Let $Y'$ be the minimal multiprojective space containing $S\setminus U$.
By Proposition
\ref{x1} we have
$w(S\setminus U)\le \binom{z}{2}+z$. Since $h^1(\Ii _{S\setminus U}(1,\dots ,1)) =h^1(\Ii _S(1,\dots ,1)) -\# (U)$. Thus
$\langle \nu (S\setminus U)\rangle =\langle S\rangle$. Thus $\nu (S)\subseteq \langle \nu (Y')\rangle$. By concision
(\cite[3.1]{l}) we have
$S\subset Y'$. Thus $w(S)=w(S\setminus U)$.
\end{proof}

\section{The examples}\label{Sk}

\begin{example}\label{k2}
Fix an integer $k\ge 2$ and integers $n_1,n_2\in \{1,2\}$. We take $Y =\PP^{n_1}\times \PP^{n_2}\times (\PP^1)^{k-2}$. Take $o
=(o_1,\dots ,o_k)\in Y$ and
$p =(p_1,\dots ,p_k)\in Y$ such that $p_i\ne o_i$ for all $i$. Take $u=(u_1,\dots ,u_k)\in Y$, $v= (v_1,\dots ,v_k)\in Y$,
$w =(w_1,\dots ,w_k)\in Y$ and $z =(z_1,\dots ,z_k)\in Y$ such that $u _i=v_i=o_i$ for all $i\ne 1$, $w_i=z_i=p_i$ for all
$i\ne 2$,
$\#\{u_1,v_1,o_1,p_1\} = \#\{o_2,p_2,w_2,z_2\}=4$. We also require if $n_1=2$ (resp. $n_2=2$)
that $\langle u_1,v_1,o_1\}\rangle \subset \PP^2$ is a line not containing $p_1$ (resp.  $\langle w_2,z_2,p_2\}\rangle
\subset
\PP^2$ is a line not containing $o_2$.  Set $S:= \{o,p,u,v,w,z\}$. By construction
$\#S =6$,
$S$ is concise for
$Y$, and $e(S)=2$. It is easy to check that $e(S')=1$ (but $S'$ is not a circuit) for any $S'\subset S$ such that $\#S'=5$.

The family of these sets $S$ has dimension $n_1+n_2+2$. If $k>2$ instead of taking the first and second factors of $Y$ we may take two arbitrary (but distinct) factors and obtain
another family of sets $S$ not projectively equivalent to the one constructed using the first and second
factors. A small modification of the construction works even if $o_i=p_i$ for some $i\in \{1,2\}$, but in that case we are forced to take $n_i:=1$.
\end{example}

\begin{example}\label{k3}
Fix integers $n\in \{1,2,3\}$ and $k\ge 1$. Set $Y:= \PP^n\times (\PP^1)^{k-1}$. If $k>1$ fix any $o_i,p_i \in \PP^1$, $2\le
i\le k$, such that $o_i\ne p_i$ for all $i$. Fix lines $L\subseteq \PP^n$ and $D\subseteq \PP^n$. If $n=2$ assume $L\ne D$.
If $n=3$ assume $L\cap D=\emptyset$. Fix $3$ distinct points $o_1,u_1,v_1\subset L$ and $3$ distinct points $w_1,p_1,z_1$ of
$D$.
If $n=1$ assume $\#\{o_1,p_1,u_1,v_1,w_1,z_1\}=6$. If $n=2$ assume $L\cap D\notin \{o_1,p_1,u_1,v_1,w_1,z_1\}$. Set $o:=
(o_1,o_2,\dots ,o_k)$, $u:=
(u_1,o_2,\dots ,o_k)$, $v:= (v_1,o_2,\dots ,o_k)$, $p:= (p_1,p_2,\dots ,p_k)$, $w:= (w_1,p_2,\dots ,p_k)$,  $z:=
(z_1,p_2,\dots,p_k)$, $A:= \{o,u,v\}$, $B:= \{p,w,z\}$, and $S:= A\cup B$. The decomposition $S =A\cup B$ immediately gives
that $S$ is equally dependent. If $k=1$ we have $e(S)=5-n$. Now assume $k>1$. Since neither $\nu(A)$ nor $\nu (B)$ are linearly
independent and $A\cap B=\emptyset$, we have $e(S)\ge 2$. Take
$D\in |\Ii _p(\epsilon _2)|$. By construction we have $S\cap D=B$. Thus the residual exact sequence of $D$ gives the exact
sequence
\begin{equation}\label{eqk1}
0 \to \Ii _A(\hat{\epsilon}_2) \to \Ii _S(1,\dots ,1)\to \Ii _{B,D}(1,\dots ,1)\to 0
\end{equation}
It is easy to check that $h^1(\Ii _A(\hat{\epsilon}_2))=1$ and that $h^1(D,\Ii _{B,D}(1,\dots ,1))=1$. Thus \eqref{eqk1} gives
$e(S)\le 2$. Thus $e(S)=2$. A small modification of the construction works even if $o_1=p_1$, but in this case we take $n<3$.
\end{example}

\begin{example}\label{k4}
Assume $k>1$. Fix $n\in \{1,2,3\}$ and an integer $s\ge 6$ and set $Y:= \PP^n\times (\PP^1)^{k-1}$. We mimic the proof of
Example
\ref{k3} taking
$3$ points on $L$ and $s-3$ points on $Y\setminus L$. We get $S\subset Y$ concise for $Y$ and such that $\#S=s$, $e(S) =s-4$ and
$e(S')<e(S)$ for all $S'\subsetneq S$. We get examples similar to Examples \ref{k2} taking instead of two points a fixed set $S'$ of points and get
a set with $\#S'+2$ points.
\end{example}

\begin{example}\label{z1}
Take $Y =\PP^2$. Fix a line $L\subset \PP^2$, any $E\subset L$ such that $\#E=3$ and a general $G\subset \PP^2\setminus L$
such that $\#G = 2$. Set $S:= E\cup G$. We have $e(S)=2$ and for any $p\in E$, the set $S\setminus \{p\}$ is a circuit. However,
$E$ shows that $S$ is not uniformly distributed.
\end{example}

\section{$4\le \#S\le 5$}\label{Sf}
In this paper we often use two results from \cite{b} which give a complete classifications of circuits $S$ with $\#S \le 5$ (\cite[Theorem 1.1 and Proposition 5.2]{b}). In this section we extend them to the case of equally dependent subsets $S\subset Y$ with $e(S)\ge 2$. Sometimes we will use them later, but the key point is that the case with arbitrarily large width and fixed $s:= \#S$ occurs exactly when $s\ge 6$.
We always call $Y =\PP^{n_1}\times \cdots \times \PP^{n_k}$ the minimal Segre variety containing $S$.

Fix a set $S\subset Y$ such that $\#S\le 5$, $e(S')<e(S)$ for all $S'\subsetneq S$
and $e(S)\ge 2$. We put the last assumption because we described all circuits (i.e. the case $e(S)=1$) in
\cite[Proposition 5.2]{b} (case $\#S=4$) and \cite[Theorem 1.1]{b} (case $\#S =5$).

Now the two new observations for the case $e(S)\ge 2$. We always assume that $S$ is concise for
$Y$.

\begin{remark}\label{f1}
Assume $\#S=4$ and $e(S)\ge 2$. By Lemma \ref{o4.1} we have $e(S)=2$ and $Y=\PP^1$. Any union $F$ of $4$ distinct points of $\PP^1$
has $e(F)=2$ and it is equally dependent.
\end{remark}

\begin{remark}\label{f2}
Assume $\#S =5$. If $e(S)\ge 3$, then $e(S)=3$, $Y=\PP^1$ and $S$ is an arbitrary subset of $\PP^1$ with cardinality $5$ (Lemma \ref{o4.1}).

Assume $e(S)=2$. Thus for all $o\in S$ we have $e(S\setminus \{o\})=1$. Let $S_o\subseteq S\setminus \{o\}$ the minimal subset
with $e(S_o)=1$. Each $S_o$ is a circuit. Let $Y[o]\subseteq Y$ be the minimal multiprojective subspace containing $o$. 
The plane $\langle \nu (S)\rangle$ contains at least $5$ points of $\nu(Y)$. Since $\nu (Y)$ is cut out by quadrics
either $\langle \nu (S)\rangle \subseteq \nu (Y)$ (and hence $Y =\PP^2$ by the assumption that $Y$ is the minimal
multiprojective space) or $\langle \nu (S)\rangle \cap\nu (Y)$ is a conic. In the latter case the conic may be smooth or reducible,
but not a double line. In this case $Y =\PP^1\times \PP^1$. To show that this case occurs we take an element $C\in |\Oo _{\PP^1\times \PP^1}(1,1)|$
and take a union $S$ of $5$ points of $C$, with no restriction if $C$ is irreducible, with the restriction that no component of $C$ contains
$4$ or $5$ points of $S$.
\end{remark}

In the last part of this section we classify the quintuples $(W,Y,q,A,B)$, where $W$ and $Y$ are multiprojective spaces, $Y\subseteq W$, $q\in \langle \nu (Y)\rangle$, $r_{\nu (Y)}(q)=2$, $A\in \Ss (Y,q)$, $B\subset W$ and $B\in \Ss (W,q,3)$. We assume that $q$ is concise for $Y$. By \cite[Proposition 3.2.2.2]{l} this assumption is equivalent to the conciseness of $A$ for $Y$. We assume that $B$ is concise for $W$, but we do not assume $W=Y$. Since $Y$ is concise for $A$ and $\#A=2$, we have $Y = (\PP^1)^k$ for some $k>0$. Since $r_{\nu(Y)}(q)\ne 1$, we have $k\ge 2$.
Since $W$ is concise for $B$ and $\#B=3$ we have $W =\PP^{m_1}\times \cdots \PP^{m_s}$ for some $s\ge k$ and $m_i\in \{1,2\}$ for all $i=1,\dots ,s$.
We see the inclusion $Y \subseteq W$, fixing for $i=1,\dots ,k$ a one-dimensional linear subspace $L_i\subseteq \PP^{m_i}$ and for $i=k+1,\dots ,s$ fixing $o_i\in \PP^{m_i}$.

We prove the following statement.

\begin{proposition}\label{cp1}
Fix $q\in \PP^r$ with rank $2$ and take a multiprojective space $Y=(\PP^1)^k$ concise for $q$. Take a multiprojective space $W\supseteq Y$ and assume the existence of $B\in \Ss (W,q,3)$. Fix $A\in \Ss (Y,q)$. Then one of the following cases occurs:
\begin{enumerate}
\item $A\cap B  \ne \emptyset$, $B$ is
obtained from
$A$ making and elementary increasing and either $W=Y$ or $W\cong \PP^2\times (\PP^1)^{k-1}$ or $W\cong (\PP^1)^{k+1}$;

\item $A\cap B =\emptyset$; in this case either $W \cong \PP^2\times \PP^1$ or $W\cong \PP^1\times \PP^1$ or $W\cong \PP^1\times \PP^1\times \PP^1$.\end{enumerate}
\end{proposition}

The multiprojective spaces $W$'s listed in part 2 of Proposition \ref{cp1} are the ones with $k>1$ in the list of \cite[Theorem 1.1]{b}.

For more on the possibles $B$'s in case (1), see Lemma \ref{cp4}. 
For the proof of Proposition \ref{cp1} we set $S:= A\cup B$. Our working multiprojective space is $W$ and cohomology of ideal sheaves is with respect to $W$. Since $\nu (A)$ and $\nu (B)$ irredundantly spans $q$, we have $e(S)>0$. Note that $k>1$, because we assumed that the tensor $q$ has tensor rank $\ne 1$.

\begin{lemma}\label{cp0}
If $A\cap B =\emptyset$, then $S$ is irredundantly dependent and either $e(S)=1$ or $e(S)=2$, $Y =\PP^1\times \PP^1$ and $S$ is formed by $5$ points of some $C\in |\Oo _{\PP^1\times \PP^1}(1,1)|$.
\end{lemma}

\begin{proof}
Since $A\cap B =\emptyset$, we have $e(S) -1 =\dim \langle \nu (A)\rangle \cap \langle \nu (B)\rangle$.
Since $\nu (A)$ (resp. $\nu (B)$) irredundantly spans $q$, we have $\langle \nu (A\setminus \{a\})\rangle \cap \langle \nu (B)\rangle\subsetneq  \langle \nu (A)\rangle \cap \langle \nu (B)\rangle$ for all $a\in A$ and $\langle \nu (A)\rangle \cap \langle \nu (B\setminus \{b\})\rangle\subsetneq  \langle \nu (A)\rangle \cap \langle \nu (B)\rangle$ for all $b\in B$.
Thus $e(S')<e(S)$ for all $S'\subsetneq S$ by the Grassmann's formula.

Assume $e(S)\ge 2$. Since $k>1$ we have $e(S)=2$ (Lemma \ref{o4.1}). Since $e(S)=2$, Remark \ref{f2} gives $W =Y =\PP^1\times \PP^1$
and that $S$ is formed by $5$ points of any smooth $C\in |\Oo _{\PP^1\times \PP^1}(1,1)|$.
\end{proof}

\begin{lemma}\label{cp4}
If $A\cap B\ne \emptyset$, then  $B$ is obtained from $A$ making an elementary increasing of $A$ with respect to the point $A\setminus A\cap B$ and one of the coordinates. In this case for any $Y =(\PP^1)^k$ concise for $q$ the concise $W$ for $B$ is either $Y$ or isomorphic to $\PP^2\times (\PP^1)^{k-1}$ in which we may prescribed which of the $k$ factors of $W$ has dimension $2$. For any rank $2$ point $q\in \langle \nu (Y)\rangle$, any $A\in \Ss (Y,q)$, any point $a\in A$ and any $i\in \{1,\dots ,k\}$ we get
a $2$-dimensional family of such sets $B$'s with $W= Y$ and a $3$-dimensional family of such $B$'s with $\dim W =\dim Y+1$.
\end{lemma}

\begin{proof}
Assume $A\cap B \ne \emptyset$. Since $\nu (A)$ and $\nu (B)$ irredundantly span $q$, we have $A\nsubseteq B$. Thus $A\cap B\subsetneq A$. Assume $A\cap B =\{o\}$ with $A =\{o,p\}$. Thus $\#S =4$.  Since $q\ne \nu (o)$, and $q\in \langle \nu (B)\rangle$, we get $\langle \nu (B)\rangle \supset \langle \nu (A)\rangle$ and in particular $\nu({p})\subset \langle \nu (B)\rangle$. 

First assume that $S$ is equally dependent.
Since $S$ is equally dependent and $s\ge k\ge 2$, by Remark \ref{f1} and \cite[Proposition 5.2]{b} we get $W=Y =\PP^1\times \PP^1$ and the list of all possible $S$'s. In this list $\nu ({p})\notin \langle \nu (S\setminus \{p\}\rangle$, a contradiction. 

Now assume that $S$ is not equally dependent. The proof of Lemma \ref{cp0} gives that $e(S')=e(S)$ only if $S' = S\setminus \{o\}$. Since $\#S' = 3$, there is $i\in \{1,\dots ,s\}$
such that $\#\pi _h(S')=1$ for all $h\ne i$. We see that $B$ is obtained from $A$ keeping $o$ and making an elementary increasing to $p$ to get two other points of $B$.
\end{proof}

\section{$\nu(A)$ or $\nu (B)$ linearly dependent}\label{So}

In this section we assume 
that at least one among $\nu (A)$ and $\nu(B)$ is linearly dependent, while in the next sections we will always assume that
both $\nu (A)$ and $\nu(B)$ are linearly independent. Just to fix the notation we
assume
$e(A)>0$. Thus
$\nu (A)$ is the union of
$3$ collinear points. By the structure of the Segre variety there is $i\in \{1,\dots ,k\}$ such that $\#\pi _h(A)=1$ for all
$h\ne i$ and
$\pi _i(A)$ is formed by the points spanning a line (Remark \ref{z3}). With no loss of generality we may assume $i=1$.

\begin{remark}\label{o1} Assume also $e(B) >0$. We want to prove that we are in one of the cases described
in Examples \ref{k2} or \ref{k3}, up to a permutation of the factors of $Y$ (assuming obviously $k>1$). By Remark \ref{z3} there
is $j\in \{1,\dots ,k\}$ such that $\#\pi _h(B)=1$ for all $h\ne j$ and $\pi_j(B)$ is formed by $3$ collinear points. 

\quad (a) Assume $i\ne j$. Up to a permutation of the factors we may assume $i=1$ and $j=2$. Fix $o =(o_1,\dots ,a_k)\in A$
and $p = (p_1,\dots ,p_k)\in B$. Set $\{u_1,o_1,v_1\}:= \pi _1(A)$ and $\{w_2,z_2,o_2,p_2\}:= \pi _2(B)$. Since $\#\pi _i(A)
=1$ for all $i>1$, we hence $\pi _i(a)=o_i$ for all $a\in A$ and all $i>1$. Since $\#\pi _i(B)
=1$ for all $i\ne 1$, we hence $\pi _i(b)=p_i$ for all $b\in B$ and all $i\ne 1$.
Thus we are as in Example \ref{k2}.

\quad (b) Now assume $i=j$. Up to a permutation of the factors we may assume $i=1$. In this case we are in the set-up of
Example
\ref{k3}.
\end{remark}

\begin{remark}\label{o2} Now assume $e(B)=0$. Since $A\subsetneq S$ and $e(A) >0$, we
have
$e(S)\ge 2$. Take
$i\in \{1,\dots,k\}$ as in part (a) and set $\{o_i\}:= 
\pi _i(A)$. By assumption $\langle \nu (B)\rangle$ is a plane and either $\langle \nu (B)\rangle \cap \langle \nu (A)\rangle =\emptyset$ (i.e. $e(S)=2$)
or $\langle \nu (B)\rangle \cap \langle \nu (A)\rangle$ is a point (call it $q'$)  (i.e. $e(S)=3$) or $\langle \nu (B)\rangle \supset \langle \nu (A)\rangle$ (i.e. $e(S)=4$).
In the latter case we have $Y=\PP^1$ (Lemma \ref{o4.1}).  Take any $A_1\subset A$ such that $\#A_1=2$ and set $S_1:= A_1\cup B$. We have $e(S_1) =e(S)-1$
and $e(S') < e(S_1)$ for any $S'\subsetneq S_1$. The set $S_1$ is very particular, because it contains a subset $A_1$ such that $\#A_1=2$ and $\#\pi _i(A) =1$
for $k-1$ integers $i\in \{1,\dots ,k\}$, say for all $i\ne 1$.

\quad (a) Assume $e(S)=3$ and hence $e(S_1) =2$. We may apply Remark \ref{f2} to this very particular $S_1$. Either $Y=\PP^2$ or $Y =\PP^1\times \PP^1$. The case $Y =\PP^2$ may
obviously occur (take $6$ points, $3$ of them on a line). To get examples with $Y=\PP^1\times \PP^1$ we need $S\subset C\in
|\Oo _{\PP^1\times \PP^1}(1,1)|$, because $e(S)=3$. The existence of $A$ gives $C$ reducible say $C =L\cup D$ with $L\in |\Oo
_{\PP^1\times \PP^1}(1,0)|$ and $D\in |\Oo _{\PP^1\times \PP^1}(0,1)|$ with $D\supset A$. Since $h^1(\Ii _B(1,1)) =0$, we see
that $\#B\cap L=2$, $\#B\cap D=1$ and $B\cap D\cap L=\emptyset$.

\quad (b) Now assume $e(S) =2$. Thus $e(S_1)$ is a circuit and we may use the list in \cite[Theorem 1.1]{b}. Hence $k\le 3$, $k=3$ implies $Y =\PP^1\times \PP^1\times \PP^1$, while $k=2$ implies $n_1+n_2\in \{2,3\}$. Obviously the case $k=1$, $Y
=\PP^3$  occurs ($6$ points of $\PP^3$ with the only restriction that $3$ of them are collinear).

\quad (b1) Assume $Y =\PP^2\times \PP^1$. We are in the set-up of \cite[Example 5.7]{b}, case $C =T_1\cup L_1$ with $L_1$ a line and $\#L_1\cap S_1 =2$. This case
obviously occurs (as explained in \cite[last 8 lines of Example 5.7]{b}). To get $S$ just add another point on $L_1$.

\quad (b2) Assume $Y= \PP^1\times \PP^1$. Here we may take as $S_1$ (resp. $S$) the union of $2$ (resp. $3$) points on any $D\in |\Oo _{\PP^1\times \PP^1}(0,1)|$
and $3$ sufficiently general points of $Y$.

\quad (b3) Assume $Y= \PP^1\times \PP^1\times \PP^1$. It does not occur here (it occurs when $e(A)=e(B)=0$ and $r_Y(q) =3$), because $\# L\cap C \le 1$ for every integral curve $C\subset \PP^1\times\PP^1\times \PP^1$ with multidegree $(1,1,1)$ and each
curve $L\subset \PP^1\times \PP^1\times \PP^1$ such that $\nu (L)$ is a line and we may apply \cite[part ({c}) of Lemma
5.8]{b}.
\end{remark}

\section{$r_Y(q)=1$}\label{Su}
We recall that in sections \ref{Su}, \ref{Sd}, and \ref{St} we assume $e(A)=e(B)
=0$ and $k>1$. In this section we assume
$r_Y(q) =1$. Take
$o\in Y$ such that
$\nu(o)=q$ and write $o =(o_1,\dots ,o_k)$.   Set
$A' =A\cup
\{o\}$ and
$B':= B\cup
\{o\}$. 

\quad (a) Assume $o\in A$. Since $\nu (o)$ is general in $\langle \nu (A)\rangle \cap \langle \nu(B)\rangle$ and $A$ has finitely many points, we have $\langle \nu (A)\rangle \cap \langle \nu (B)\rangle =\{\nu (o)\}$.
The Grassmann's formula gives $\dim \langle \nu (S)\rangle =4$, i.e. $e(S)= 1$. Since $A\cap B=\emptyset$, we have $o\notin B$. Thus $\nu (B\cup \{o\})$ is linearly dependent.
Since $B\cup \{o\}\subsetneq S$, $e(S)=1$ and $S$ is assumed to be equally dependent, we get a contradiction. In the same way we prove that $\#B'=4$.

\quad (b) By step (a) we have $\# A'=\#B' =4$. Write $o =(o_1,\dots ,o_k)$. The sets $\nu(A')$ and $\nu(B')$ are linearly dependent. Assume for the moment the existence of $A''\subsetneq A'$ such
that $e(A'') =e(A')$. We have $\#A''=3$, $e(A'')=1$ and there is $i\in \{1,\dots ,k\}$ such that $\#\pi _h(A'') =1$ for all $h\ne 1$. Since $e(A)=0$, $o\in A''$. Set $\{b\} := A\setminus A\cap A'$. We see that $A$ is obtained from $\{o,b\}$ making an elementary increasing with respect to $o$ and the $i$-th factor. But then $\nu (o)$ is spanned by $\nu(A\cap A'')$, contradicting the generality of $q\in \langle \nu (A)\rangle \cap \langle \nu (B)\rangle$ and that $S$ is equally dependent.
In the same way we handle the case in which there is $B''\subsetneq A'$ such that $\nu (A'')$ is dependent.

\quad ({c}) By steps (a) and (b) we may assume that $\nu(A')$ and $\nu(B'))$ are circuits. Let $Y' =\prod _{i=1}^{s}
\PP^m_i$
(resp. $Y'' = \prod _{i=1}^{c} \PP^{t_i}$) be the minimal multiprojective subspace of $Y$ containing $A'$ (resp. $B'$). 
 By \cite[Proposition 5.2]{b} either $s=1$ and $m_1=2$ or $s=2$ and $m_1=m_1=2$, either $c=2$ and $t_1=2$ or $c=2$ and
$t_1=t_2=1$.

\quad (c1) Assume $s=c=2$. Up to a permutation of the factors we may assume $\#\pi _h(A')=1$ for all $h>1$. Call $1\le i<j\le k$ the two indices such that
$\#\pi _h(B') =1$ for all $h\notin \{i,j\}$. Note that $\pi _h(S) =\pi _h(o)$ if $h\notin \{1,2,i,j\}$.

\quad {\bf Claim 1:} $k=j$.

\quad {\bf Proof Claim 1:} Assume $k>j$. Since $k>j\ge 2$, we have $\pi _k(A) =\pi _k(o) =\pi _k(B)$. Thus $Y$ is not concise for $Y$.

\quad {\bf Claim 2:} $k\le 4$ and $Y=(\PP^1)^4$ if $k=4$.

\quad {\bf Proof Claim 2:} By Claim 1 we have $k\le 4$. Assume $k=4$, i.e. assume $i=3$ and $j=4$. Assume $Y\ne (\PP^1)^4$, i.e. assume $n_h\ge 2$ for some $h$, say for $h=1$. Fix $a\in A$. Since $h^0(\Oo _Y(\epsilon _1)) =n_1+1\ge 3$, there is $H\in |\Oo _Y(\epsilon _1)|$ containing $o$ and at least one point of $B$. By concision we have $S\nsubseteq H$.
Since $A$ and $B$ irredundantly spans $q$, \cite[5.1]{bb2} or \cite[2.4, 2.5]{bbcg1} gives $h^1(\Ii _{S\setminus S\cap H}(0,1,1,1)) >0$. Since  $\#\pi _1(B') =1$, we have $B\subset H$.
Thus $\#(S\setminus S\cap H)\le 2$. Since $\Oo _Y(\epsilon _1)$ is globally generated, we get $\#(S\setminus S\cap H) =2$ (i.e. $S\setminus S\cap H =A\setminus \{a\}$).
Since $\Oo _{Y_1}(1,1,1)$ is very ample, we get $\#\eta _1(A\setminus \{a\}) =1$. Taking another $a'\in A$ instead of $a$, we get $\#\eta _1(A)=1$, i.e. $A$ does not depend on the second factor of $Y$. Since $\nu(A)$ irredundantly spans $\nu(o)$, we get $\#\pi _1(A')=1$, a contradiction.

 \quad (c2) Assume $s=2$ and $c=1$ (the case $s=2$ and $c=1$) being similar. We may assume $\pi _h(A')=1$ for all $h>2$. Call $i$ the only index such that $\#\pi _i(B')>1$.
 As in step (c1) we get $k \le \#\{1,2,3\}\le 3$.
 
 \quad (c3) Assume $s=c=1$. As is step (c1) and (c2) we get $k\le 2$.

\section{$r_Y(q)=2$}\label{Sd}

In this section we assume $r_Y(q)=2$. We fix $E\in \Ss (Y,q)$. Set $M:= \langle \nu (A)\rangle \cap \langle \nu (E)\rangle$. Call $Y'$ (resp. $Y''$) the minimal multiprojective subspace containing $E\cup A$ (resp. $E\cup B)$)

\begin{lemma}\label{g1}
Take a circuit $F\subset Y:= \PP^1\times \PP^1\times \PP^1$ concise for $Y$ and with $\#S=5$. Write $F =E\cup A$ with $\#E=2$
and $\#A=3$. Then $Y$ is concise for $E$.
\end{lemma}

\begin{proof}
By \cite[Lemma 5.8]{b} $F$ is contained in an integral curve $C\subset Y$ of tridegree $(1,1,1)$. Each map $\pi
_{i|C}: C\to \PP^1$ is an isomorphism. Thus each $\pi _{i|S}$ is injective.\end{proof}

\begin{remark}\label{g2}
Since $q$ has tensor rank $3$, each set $\nu (A)$ and $\nu (B)$ irredundantly spans $q$.
\end{remark}

\begin{lemma}\label{d2}
If  $E\cap A\ne \emptyset$ (resp. $E\cap B\ne \emptyset$) only if either $w(S)\le 3$ or $A$ (resp. $B$) is obtained form $E$ making an
elementary increasing.  
\end{lemma}

\begin{proof}
It is sufficient to prove the lemma for the set $A$. The `` if '' part follows from the definition of elementary
transformation, because $\#E >1$.

Assume $E\cap A \ne \emptyset$. Since $\nu (A)$ irredundantly spans $q$ (Remark \ref{g2}), we have
$E\nsubseteq A$. Write $E\cap A=\{a\}$, $E =\{a,b\}$ and $A = \{a,u,v\}$. We need to prove that there is $i$ such that
$\pi _h(a)=\pi _h(u)=\pi _h(v)$ for all $h\ne i$, while $\pi _i(\{a,u,v\})$ spans a line. 

\quad (a) First assume that $E\cup A$ is not equally dependent. Since $\#(E\cup A)=4$, we have $e(E\cup A)=1$ and there is $F\subset E\cup A$ such that $\#F =3$ and $e(F)=1$.
By Remark \ref{z3} there is $i$ such that $\#\pi _h(F) =1$ for all $h\ne i$ and $\pi _i(F)$ is formed by $3$ collinear points. Since $\nu (E)$ and $\nu (A)$ irredundantly spans
$q$ (Remark \ref{g2} and the assumption $E\in \Ss (Y,q)$), it is easy to check that $(E\cup A)\setminus F =\{a\}$. Thus $A$ is obtained from $E$ applying an elementary increasing with respect to $b$ and the $i$-th factor of the multiprojective space.

\quad (b) Now assume that $E\cup A$ is equally dependent. Since $\#(E\cup A)=4$, \cite[Proposition 5.2]{b} says that $w(E\cup A)\le 2$
and that $\PP^1\times \PP^1$ is the minimal multiprojective space containing $E\cup A$. Since $E\in \Ss (Y,q)$ and $rY(q)>0$, $Y'\cong \PP^1\times \PP^1$ is the minimal multiprojective space containing $E$.

\quad (b1) Assume $E\cap B \ne \emptyset$ and $E\cup B$ is not equally dependent. By step (a) applied to $B$ we get that $B$ is obtained from $E$ making a positive elementary transformation. Thus either $w(B) =2$ or $\PP^1\times \PP^1\times \PP^1$ is the minimal multiprojective space containing $B$ (last sentence of Example \ref{k2}) and it contains $A$, too, since it contains $E$. Thus $w(S)\le 3$.

\quad (b2) Assume $E\cap B\ne \emptyset$ and $E\cup B$ is equally dependent. Thus $Y''\cong \PP^1\times \PP^1$ and $Y''$ is the minimal multiprojective subspace containing $E$. Hence $Y''=Y'$ and $Y=\PP^1\times \PP^1$.

\quad (b3) Assume $E\cap B =\emptyset$. We get $w(Y'') \le 3$ by Proposition \ref{cp1} and (since $W\supseteq Y'$) we get
$Y=W$.\end{proof}

\begin{lemma}\label{d2.0}
Assume $E\cap A\ne \emptyset$ and $E\cap B\ne \emptyset$. Then either $w(S)\le 2$ or $S$ is as in one of the Examples \ref{k2} or \ref{k3}.
\end{lemma}

\begin{proof}
Assume $w(S)>2$. By Lemma \ref{d2} $A$ and $B$ are obtained from $E$ making an elementary increasing. Since $A\cap B =\emptyset$, we have $\# A\cap E =\# B\cap E=1$ and $E\subset S$. By the definition of elementary
transformation it is obvious that $S$ is as in one of the Examples \ref{k2}  or \ref{k3} (Example \ref{k3}  occurs
if and only if we are doing the elementary increasings giving $A$ and $B$ from $E$ with respect to the same factor of the multiprojective space).
\end{proof}

\begin{lemma}\label{d1.1}
Assume $E\cap A= \emptyset$ (resp. $E\cap B=\emptyset$). Then $E\cup A$ (resp. $E\cup B$) is equally dependent.
\end{lemma}

\begin{proof}
It is sufficient to prove the lemma for $E\cup A$. The assumption is equivalent to $\dim M = e(E\cup A)-1$. Fix $a\in A$. Since $q\notin \langle \nu (A\setminus \{a\})\rangle$, we have $\langle \nu (A\setminus \{a\})\rangle \cap \langle \nu (E)\rangle\subsetneq M$. The Grassmann's formula gives $e((E\cup A)\setminus \{a\}) < e(E\cup A)$. Take $b\in E$. Since $q\notin \langle \nu (E\setminus \{b\})\rangle$, we have $\langle \nu (E\setminus \{b\})\rangle \cap \langle \nu (A)\rangle\subsetneq M$. Thus $E\cup A$ is equally dependent.
\end{proof}

\begin{lemma}\label{g4}
Assume $E\cap A =E\cap B=\emptyset$. Then $w(S) \le 3$ and $Y\cong \PP^1\times \PP^1\times \PP^1$ if $w(S)=3$.
\end{lemma}

\begin{proof}
By Proposition \ref{cp1} and Lemmas \ref{g1} and \ref{d1.1} we have $w(Y')\le 3$, $w(Y'')\le 3$ and if one of them,  say $Y'$, is $3$, then $Y' \cong \PP^1\times \PP^1\times \PP^1$ and $\PP^1\times \PP^1\times \PP^1$ is the
minimal multiprojective space containing $E$. Hence $w(Y'') =3$ and $Y'=Y''$, i.e. $Y\cong \PP^1\times \PP^1\times \PP^1$. Now assume $w(Y') =w(Y'') =2$.
In this case both $Y'$ and $Y''$ have the same number of factors as the minimal multiprojective space containing $E$ and exactly the same factor, i.e. if $E =\{u,v\}$
with $u=(u_1,\dots ,u_k)$, $v =(v_1,\dots ,v_k)$ and $u_i=v_i$ for all $i>2$, then $\#\pi _i(Y') =\#\pi _i(Y'') =1$ for all $i>2$. Since $\pi _i(Y')=\{u_i\} =\pi _i(Y'')$ for all $i>2$, we get $w(Y)=2$.
\end{proof}

\begin{lemma}\label{g5}
Either $S$ is as in Examples \ref{k2}, \ref{k3} or $w(S)\le 4$ with $Y=(\PP^1)^4$ if $w(S)=3$.
\end{lemma}

\begin{proof}
By the previous lemmas we may assume that exactly one among $E\cap A$ and $E\cap B$, say the first one, is empty.
Thus $B$ is obtained from $E$ making a positive elementary transformation, while $w(Y') \le 3$ and $Y'\cong \PP^1\times \PP^1\times \PP^1$. First assume
$w(Y') =\PP^1\times \PP^1\times \PP^1$. By Lemma \ref{g1} $Y'$ is the minimal multiprojective space containing $E$. Hence
$w(E\cup B)\le 4$ and $Y'' =(\PP^1)^4$ with $Y\supset Y'$ if $w(Y'')=4$ (last part of Example \ref{k2}). We get $w(Y)\le 4$
and $Y\cong (\PP^1)^4$ is $S$ is not as in Examples \ref{k2} or \ref{k3}.
Now assume $w(Y')=2$. Thus $w(E)=2$. We get that either $w(Y'')=2$ or $Y''\cong \PP^1\times \PP^1\times \PP^1$ with $\#\pi
_3(A)=1$. Hence $w(Y)\le 3$.
\end{proof}

\section{$r_Y(q)=3$}\label{St}

The point $q\in \PP^N$ has tensor rank $3$ and hence $\nu(A)$ and $\nu (B)$ are tensor decompositions of it with the minimal
number of terms. By concision (\cite[Proposition 3.2.2.2]{l}) $Y$ is the minimal multiprojective space containing $A$ and the
minimal multiprojective space containing $B$. Hence $1\le n_i\le 2$ for all $i$.  $Y$ is as in the cases of
\cite[Theorem 7.1]{bbs} coming from the cases $\#S =6$, i.e. we exclude case (6) of that list. In all cases (1), (2), (3),
(4), (5) of that list we have $w(Y)\le 4$ and $w(Y)=4$ if and only if $Y \cong (\PP^1)^4$.
 The sets
$\Ss (Y,q)$ to which
$A$ and
$B$ belong are described in the same paper. The possible concise $Y$'s are listed in \cite[Theorem 7.1]{bbs}, but we stress that from the
point of view of tensor ranks among the sets
$S$ described in one of the examples of \cite{bbs} there is some structure. If we start with $S$ with $e(S)=1$ and arising in
this section  and any decomposition
$S =A\cup B$ with $\#A=\#B =3$, the assumption $e(S)=1$ and $e(A)=e(B)=0$ gives that $\langle \nu (A)\rangle \cap \langle \nu
(B)\rangle$ is a single point by the Grassmann's formula. Call $q$ this point. If we assume $r_X(q)=3$, then in \cite{bbs}
there is a description of all $S\in \Ss(Y,q)$. Changing the decomposition $S =A\cup B$ change $q$ and hence all sets
associated to $S$ using the point $q$. Thus if $e(S)=1$ and there is a partition $S =A\cup B$ of $S$ such that tte point $\langle
\nu(A)\rangle \cap \langle \nu(B)\rangle$ has tensor rank $3$, then to $S$ and the partition $S=A\cup B$ we may associate a
family $\Ss (Y,q)$ of circuits associated to $q$.

\begin{proof}[End of the proof of Theorem \ref{is1}:]
In the last $4$ sections we considered all possible cases coming from a fixed partition of $A\cup B$. We summarized the case
$r_Y(q)=2$ in the statement of Lemma \ref{g5}.
\end{proof}

\end{document}